\documentclass[11pt]{amsart}
\newcommand{\ben}{\begin{enumerate}}
\newcommand{\een}{\end{enumerate}}
\newcommand{\eq}[2][label]{\begin{equation}\label{#1}#2\end{equation}}
\newcommand{\av}[2]{\langle #1\rangle_{_{\scriptstyle #2}}}

\newcommand{\ve}{\varepsilon}
\renewcommand{\phi}{\varphi}
\usepackage{amsmath}
\usepackage{amsfonts}
\usepackage{amssymb}
\usepackage{amsthm}
\usepackage{graphicx,color,hyperref}
\usepackage{mathtools}
\usepackage[normalem]{ulem} 

\usepackage{enumitem}

\newcommand{\BMO}{{\rm BMO}}
\newcommand{\VMO}{{\rm VMO}}
\newcommand{\Lip}{{\rm Lip}}

\newcommand{\conv}[2]{{\rm sub}_{#2}\,{#1}}

\newcommand{\rn}{\mathbb{R}^n}

\newcommand{\rr}{r}
\newcommand{\ttt}{\mathcal{T}}
\newcommand{\uuu}{U}

\newcommand{\AAA}{A}

\newtheorem{theorem}{Theorem}[section]

\newtheorem{lemma}[theorem]{Lemma}
\newtheorem{corollary}[theorem]{Corollary}
\newtheorem{prop}[theorem]{Proposition}

\newtheorem*{theorem*}{Theorem}{\bf}{\it}
\newtheorem*{proposition*}{Proposition}{\bf}{\it}
\newtheorem*{observation*}{Observation}{\bf}{\it}
\newtheorem*{lemma*}{Lemma}{\bf}{\it}
\newtheorem*{conjecture*}{Conjecture}{\bf}{\it}

\theoremstyle{definition}

\theoremstyle{remark}
\newtheorem{remark}[theorem]{Remark}

\numberwithin{equation}{section}

\allowdisplaybreaks

\DeclareMathOperator{\col}{Cut}

\newcommand{\con}{\star}
\newcommand{\ck}{\Psi}

\newcommand{\gmm}{\Gamma}
\newcommand{\gmt}[1]{\Gamma^{#1}}
\newcommand{\gt}[1]{\gamma^{#1}}

\newcommand{\TT}{T}
\newcommand{\txi}{\tilde\xi}
\newcommand{\GG}{G}
\newcommand{\tGG}{\tilde G}

\newcommand{\Xicl}[1]{{\mathcal L}_{#1}}

\title[Monotone rearrangement in VMO]{Monotone rearrangement does not increase\\ generalized Campanato norm in VMO}
\begin{document}

\author{Leonid Slavin}
\author{Pavel Zatitskii}
\address{University of Cincinnati, P.O. Box 210025, OH 45221-0025, USA}

\email{leonid.slavin@uc.edu}

\address{University of Cincinnati, P.O. Box 210025, OH 45221-0025, USA}

\email{zatitspl@ucmail.uc.edu}

\thanks{The first author is supported by the Simons Foundation, collaboration grants 317925 and 711643}
\thanks{Both authors express thanks to the Taft Research Foundation at the University of Cincinnati}

\subjclass[2020]{Primary 42A05, 42B35, secondary 26A16, 46E30}

\keywords{vanishing mean oscillation, Campanato norm, monotone rearrangements}

\begin{abstract}
We consider a quantitative version of the space VMO on an interval, equipped with a quadratic Campanato-type norm, and prove that monotone rearrangement does not increase the norm in this space.
\end{abstract}

\maketitle

\section{Introduction}
Symbols $I$ and $J$ will denote finite intervals and $\av{\,\cdot\,}J$ will stand for the average over $J$ with respect to the Lebesgue measure. Let $\xi \colon [0,\infty) \to [0,\infty)$ be a continuous, increasing function such that
$\xi(0)=0$. We will call such functions moduli, and use the same term for their restrictions to intervals of the form $[0,T].$

Fix $I$ and take a modulus $\xi$ on $[0,|I|].$ Let $\Xicl{\xi}(I)$ be the class of all functions $\varphi\colon I\to\mathbb{R}$ for which the following quantity is finite:
\eq[norm]{
\|\phi\|_{\Xicl{\xi}(I)} := \inf\big\{C>0 \colon \av{\phi^2}{J} - \av{\phi}{J}^2 \leq C^2 \xi^2(|J|), J \subset I\big\}.
}
It is easy to see that $\|\cdot\|_{\Xicl{\xi}(I)}$ is a seminorm, and $\|\phi\|_{\Xicl{\xi}(I)}=0$ if and only if $\phi$ is constant a.\,e. on $I.$ In fact, for some choices of $\xi$ the class $\Xicl{\xi}(I)$ contains only constant functions. We give the following for the sake of completeness.
\begin{theorem} 
\label{th0}
The class $\Xicl{\xi}(I)$ contains non-constant functions if and only if
$$
\liminf_{t \to 0^+} \frac{\xi(t)}{t}>0.
$$
\end{theorem}
The elementary proof of Theorem~\ref{th0} can be found in Section~\ref{non-empty}. 
\begin{remark}
We only need this theorem in the one-dimensional setting, but the proof given in Section~\ref{non-empty} works without changes in higher dimensions if one considers the sets $I$ and $J$ in definition~\eqref{norm} to be cubes in $\rn$ and the argument of $\xi$ to be the side-length of $J.$
\end{remark}

The class $\Xicl{\xi}$ was introduced by Spanne in~\cite{Spa} in a slightly different formulation (it is sometimes referred to as $\BMO_\xi$ in literature). This generalized the class $C_{p,\alpha}$ of all functions $\varphi$ satisfying the condition 
$$
\sup_{J}\,\frac1{|J|^{p\alpha}}\,\av{|\varphi-\av{\varphi}J|^p}J<\infty,
$$
with the supremum is taken over all sets $J$ in some basis, which was considered by Campanato in~\cite{Camp} and Meyers in~\cite{Mey}. For $p\ge1$ and $0<\alpha<1,$ $C_{p,\alpha}$ coincides with the Lipschitz class $\Lip_\alpha;$
for $\alpha=0$ one obtains $\BMO$ (with the so-called $\BMO^p$ norm).  For $p=2,$ the left-hand side can be written as $\frac1{|J|^{2\alpha}}(\av{\varphi^2}J-\av{\varphi}J^2).$ Thus, $\Xicl{\xi}$ coincides with $C_{2,\alpha}$ when $\xi(t)=t^\alpha.$ Furthermore, since for $\varphi\in\Xicl{\xi},$ we have  $\av{\varphi^2}J-\av{\varphi}J^2\to 0$ as $|J|\to0,$ $\Xicl{\xi}$ can be considered a quantitative version of $\VMO.$ The reader can find more information about Campanato classes in~\cite{Kisl} and about VMO in~\cite{Sar} and~\cite{Gar}.

Our main result (Theorem~\ref{th1} below) is that monotone rearrangement does not increase the norm in $\Xicl{\xi}(I).$ Take an interval $I=(a,b).$ For a function $f\colon I\to \mathbb{R},$ its decreasing rearrangement is the function $f^*$ on $I$ given by
$$
f^*(t)=\inf\{\lambda\colon |\{s\in I\colon f(s)>\lambda\}|\le t-a\},~t\in I.
$$
Thus, $f^*$ is the unique (up to sets of measure zero) decreasing function on $I$ that is equidistributed with $f$:
$$
|\{s\in I\colon f^*(s)>\lambda\}| = |\{s\in I\colon f(s)>\lambda\}|, \qquad \lambda \in \mathbb{R}. 
$$

The study of the behavior of oscillation classes under monotone rearrangement has a long history. In~\cite{Klemes}, Klemes showed that rearrangement does not increase the $\BMO^1$ norm on $(0,1).$ An exposition of rearrangement techniques in the setting of $\BMO^1$ and various weight classes defined on rectangles can be found in~\cite{Kor2}. In~\cite{sz}, the second author and Stolyarov proved that rearrangement preserves general classes $A_\Omega$ defined on an interval and, in particular, does not increase the class constant in $\BMO^2$ and $A_p.$ In higher dimensions, we note the recent study~\cite{Ryan1}, which traces the dimensional dependence in rearranging BMO functions from $\rn$ to the half-axis $[0,\infty),$ and the paper \cite{studia}, which gives the exact constant of rearrangement from BMO on $\alpha$-martingales (including the dyadic $\BMO([0,1]^n)$) to the usual BMO on an interval.
We also note the qualitative results of article~\cite{Ryan2}, where it is shown that the rearrangement is continuous as an operator from VMO to VMO.

Our interest in the subject stems mainly from the fact that our result allows one to work with monotone functions in proving sharp estimates for the class $\Xicl{\xi}(I).$ (This line of reasoning is well established for $\BMO^1$: for instance, Korenovskii~\cite{Kor2} used Klemes's theorem from~\cite{Klemes} to find the exact John--Nirenberg constant of $\BMO^1(0,1).$) In particular, it enables Bellman-function analysis in this setting. The beginnings of such analysis -- without rearrangements -- were laid by Os\c{e}kowski in~\cite{Os}; for $\xi(t)=t^\alpha$ he found the sharp constants of equivalence between $\Lip_\alpha$ and $C_{2,\alpha}$ norms, as well as the sharp upper and lower bounds on $\inf_I\varphi$ and $\sup_I\varphi$ for $\varphi\in C_{2,\alpha}.$ (An extension of the latter result is a key ingredient in our proof.)
A more general theory of estimating functionals of the type $\av{f(\varphi)}J$ in terms of $\av{\varphi}J,$ $\av{\varphi^2}J,$ and $|J|$ relies on rearrangements. In a companion paper~\cite{exp-vmo}, we present the first Bellman function of this type, in the context of sharp exponential estimates.

Here is our main theorem.
\begin{theorem}\label{th1}
Take an interval $I$ and let $\xi$ be a modulus on $[0,|I|].$ If $\phi\in\Xicl{\xi}(I),$ then
$\phi^*\in\Xicl{\xi}(I)$ and
\eq[eqmain]{
\|\phi^*\|_{\Xicl{\xi}(I)} \leq \|\phi\|_{\Xicl{\xi}(I)}.
}
\end{theorem}
One can reformulate the theorem without mentioning the class~$\Xicl{\xi}$. To that end, with every function $\phi \in \VMO(I)$ we associate a modulus $\xi_\phi$ as follows: $\xi_\phi(0)=0$ and 
\eq[eq3801]{
\xi_\phi(t)= \sup\Big\{\big(\av{\phi^2}{J}-\av{\phi}{J}^2\big)^{1/2}\colon J\subset I, 0<|J| \leq t\Big\}, \qquad 0<t\leq |I|.
}
One can easily show that $\xi_\phi$ is continuous on $[0,|I|]$ and that $\xi_\phi$ is the pointwise smallest modulus $\xi$ such that $\|\phi\|_{\Xicl{\xi}(I)}\le 1.$
Note that for a modulus $\xi$ the inequality $\|\phi\|_{\Xicl{\xi}(I)} \leq C$ is equivalent to $\xi_\phi \leq C\xi$. Therefore, the statement of Theorem~\ref{th1} can be written simply as 
\eq[3816]{
\xi_{\phi^*} \leq \xi_{\phi}.
}
We can improve this estimate with the help of the following proposition.
\begin{prop}
\label{pr00}
Let $\psi$ be a monotone function in $\VMO(I).$ Then the function $t\mapsto t^2\xi_{\psi}^2(t)$ is convex on $[0,|I|]$. 
\end{prop} 
To formulate our improvement, we need a new notation. For a fixed $T>0$ and a function $f\colon [0,T]\to[0,\infty),$ let $\conv{f}{T}$ be the largest function on $[0,T]$ such that $\conv{f}{T}\le f$ and the function $s\mapsto s^2(\conv{f}{T})^2(s)$ is convex on $[0,T]$.

Now, combining Theorem~\ref{th1} and Proposition~\ref{pr00} we obtain:
\begin{theorem}\label{ThCor}
If $\phi \in \VMO(I),$ then $\phi^* \in \VMO(I)$ and $\xi_{\phi^*} \leq \conv{\xi_\phi}{|I|}.$
\end{theorem}

We record this inequality not only because it is clearly stronger than~\eqref{3816}, but also because we need it for future use; cf.~\cite{exp-vmo}.

In the next section, we prove Theorem~\ref{th1} under additional regularity assumptions on~$\xi,$ including the key convexity assumption:
\eq[eqConv]{
\AAA\colon t\mapsto t^2\xi^2(t) \quad \text{is convex.}
}
In Sections~\ref{non-smooth} and~\ref{SecNew} we dispose of those assumptions. Finally, in Section~\ref{non-empty} we prove Theorem~\ref{th0}, Proposition~\ref{pr00}, and Theorem~\ref{ThCor}.
\begin{remark}
In the first version of the paper, Theorem~\ref{th1} was proved under assumption ~\eqref{eqConv} on $\xi.$ (Such moduli $\xi$ were called admissible in that earlier version.) However, we have since been able to strengthen the result to its current form, adapting an argument due to F.~Nazarov, which is detailed in Section~\ref{SecNew}. The convexity condition still plays an important intermediate role in the proof of Theorem~\ref{th1} and is also needed in the statement of Theorem~\ref{thAdam} in Section~\ref{smooth}, which is of independent value for Bellman-function analysis in this setting. Note that, as  Proposition~\ref{pr00} shows, condition~\eqref{eqConv} is automatically fulfilled for moduli associated with monotone VMO functions.
\end{remark}

\section{Proof of Theorem~\ref{th1} in the smooth case}
\label{smooth}
Throughout this section, we will assume that  modulus $\xi$ is twice continuously differentiable on $(0,\infty)$ and
\eq[611]{
\xi'(t)>0, \qquad A''(t)>0,\quad t>0.
}

Let us introduce some key geometric objects and summarize their properties for further use. We will need the function
\eq[3861]{
\rr(s)=
\begin{cases}
\sqrt{\xi^2(s)+2s\xi(s)\xi'(s)},&s>0,\\
0,&s=0.
\end{cases}
} 

\begin{lemma}\label{rem1}
The function $\rr$ is continuous on $[0,\infty).$
\end{lemma}
\begin{proof}
    The only possible issue is continuity at 0. Observe that since $A$ is convex, for any $t>0$ we have
$$
A'(t)\le \frac{A(2t)-A(t)}t\le \frac{A(2t)}t=4t\xi^2(2t).
$$
Since $A'(t)=t(\xi^2(t)+r^2(t))$ and $\xi(0)=0,$ the claim follows.
\end{proof}

Let 
$$
\Omega = \{(x_1,x_2,t) \in \mathbb{R}^3\colon x_1^2\leq x_2\leq x_1^2 + \xi^2(t)\}.
$$
Directly from the definition of the class $\Xicl{\xi}(I)$ we have
$$
\|\phi\|_{\Xicl{\xi}(I)} \leq 1 \qquad \Longleftrightarrow  \qquad (\av{\phi}{J}, \av{\phi^2}{J}, |J|) \in \Omega \qquad \text{ for any subinterval } J\subset I. 
$$
Let us consider the curve 
\eq[3896]{
\gmm(\tau) = (\gmm_1(\tau), \gmm_2(\tau)) = \big(\tau \rr(\tau), \tau (\rr^2(\tau)+\xi^2(\tau))\big), \qquad \tau \geq 0. 
}
For any fixed $t>0$ define two additional curves: 
$$
\gmt{t}(\tau) = \frac{1}{t}\,\gmm(\tau) =  \Big(\frac{\tau}{t} \rr(\tau), \frac{\tau}{t} (\rr^2(\tau)+\xi^2(\tau))\Big),
\qquad \tau\in[0,\infty),
$$ 
and its restriction to the interval $[0,t],$
$$
\gt{t}(\tau)= \gmt{t}(\tau), \qquad \tau \in [0,t].
$$
This also defines the component functions $\gmt{t}_1,$ $\gmt{t}_2,$ $\gt{t}_1,$ and $\gt{t}_2.$
The curve $\gt{t}$ starts at $(0,0)$, ends at $(\rr(t), \rr^2(t) + \xi^2(t))$, and splits the parabolic strip 
$$
\Omega^t:= \{(x_1,x_2) \in \mathbb{R}^2\colon x_1^2 \leq x_2 \leq x_1^2 + \xi^2(t)\}
$$  
into two parts; see Fig.~\ref{fig1}.
\begin{figure}
    \includegraphics[width=0.6\textwidth]{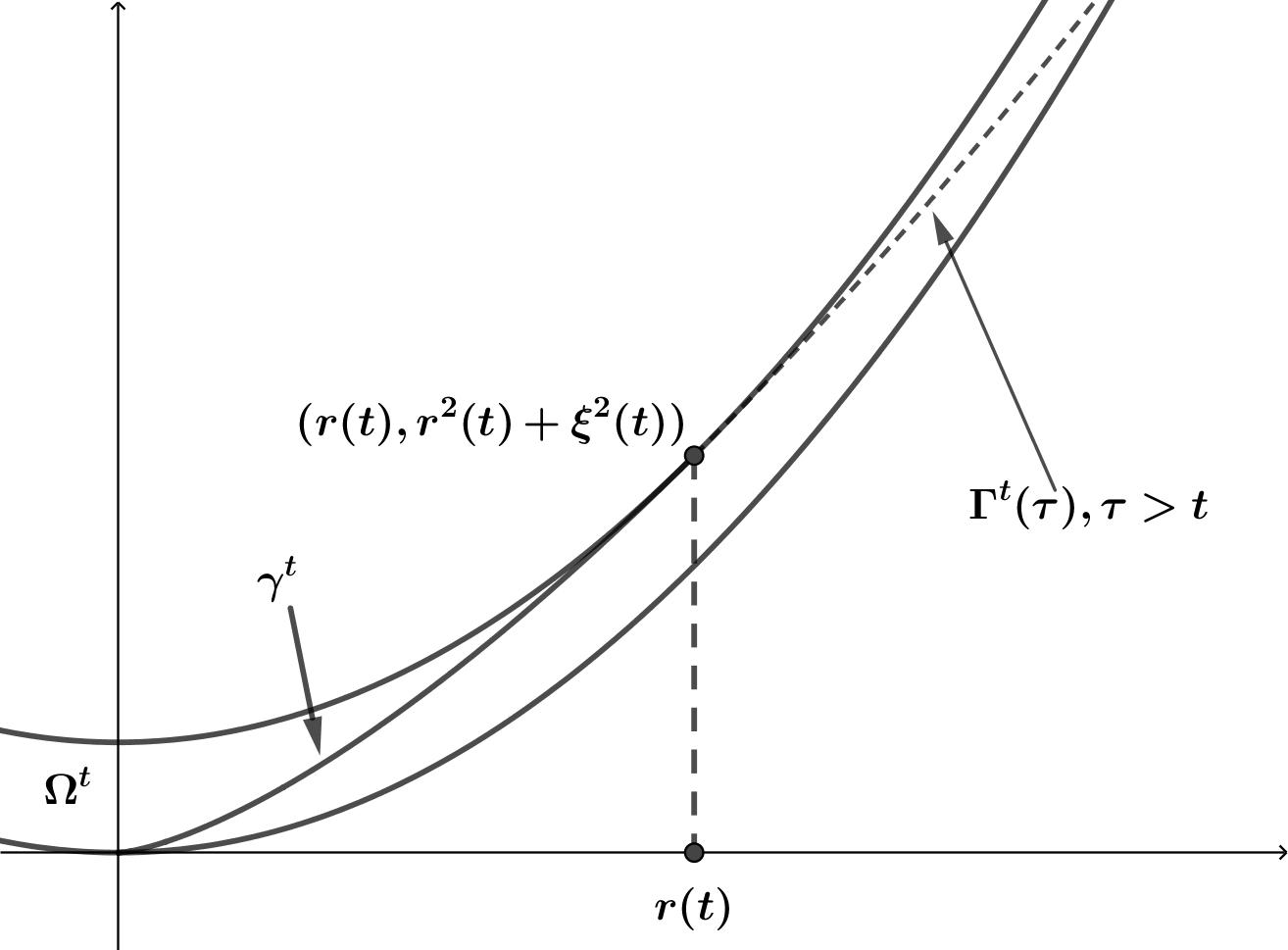}
    \caption{The curve $\gt{t}$}
    \label{fig1}
\end{figure}
\subsection{Properties of $\Gamma,$ $\Gamma^t,$ and $\gamma^t$} Here we summarize simple properties of the curves $\Gamma,$ $\Gamma^t,$ and $\gamma^t$ that will be used in the rest of the proof. 
\begin{lemma}\label{lem1}
The functions $\gmm_1,\gmm_2$ satisfy the relations:
\eq[eq1]{
\gmm_1' = \frac{\AAA''}{2\rr} > 0;
}
\eq[eq2]{
\gmm_2' = A''  >  0;
}
\eq[eq9]{
\Big(\frac{\gmm_2}{\gmm_1}\Big)'(\tau) = A''(\tau) \frac{\tau^2 \xi(\tau)\xi'(\tau)}{\rr(\tau)\gmm_1^2(\tau)}> 0.
}
\end{lemma}
The statement of the lemma follows from~\eqref{eqConv}, \eqref{3861}, and~\eqref{3896} by direct differentiation.
\begin{lemma}\label{lem2}
If $0<t_1<t_2$, then the curve $\gmt{t_1}$ lies below the curve $\gmt{t_2}$; see Figure~\ref{fig2}.
\end{lemma}
\begin{proof}
For any $\tau>0$ the point $\gmt{t_2}(\tau)$ lies on a segment connecting $\gmt{t_1}(\tau)$ with the origin. This segment lies over the curve $\gmt{t_1}$ by inequalities~\eqref{eq1} and~\eqref{eq9}.
\end{proof}
\begin{figure}[t]
    \centering
    \includegraphics[width=0.38\textwidth]{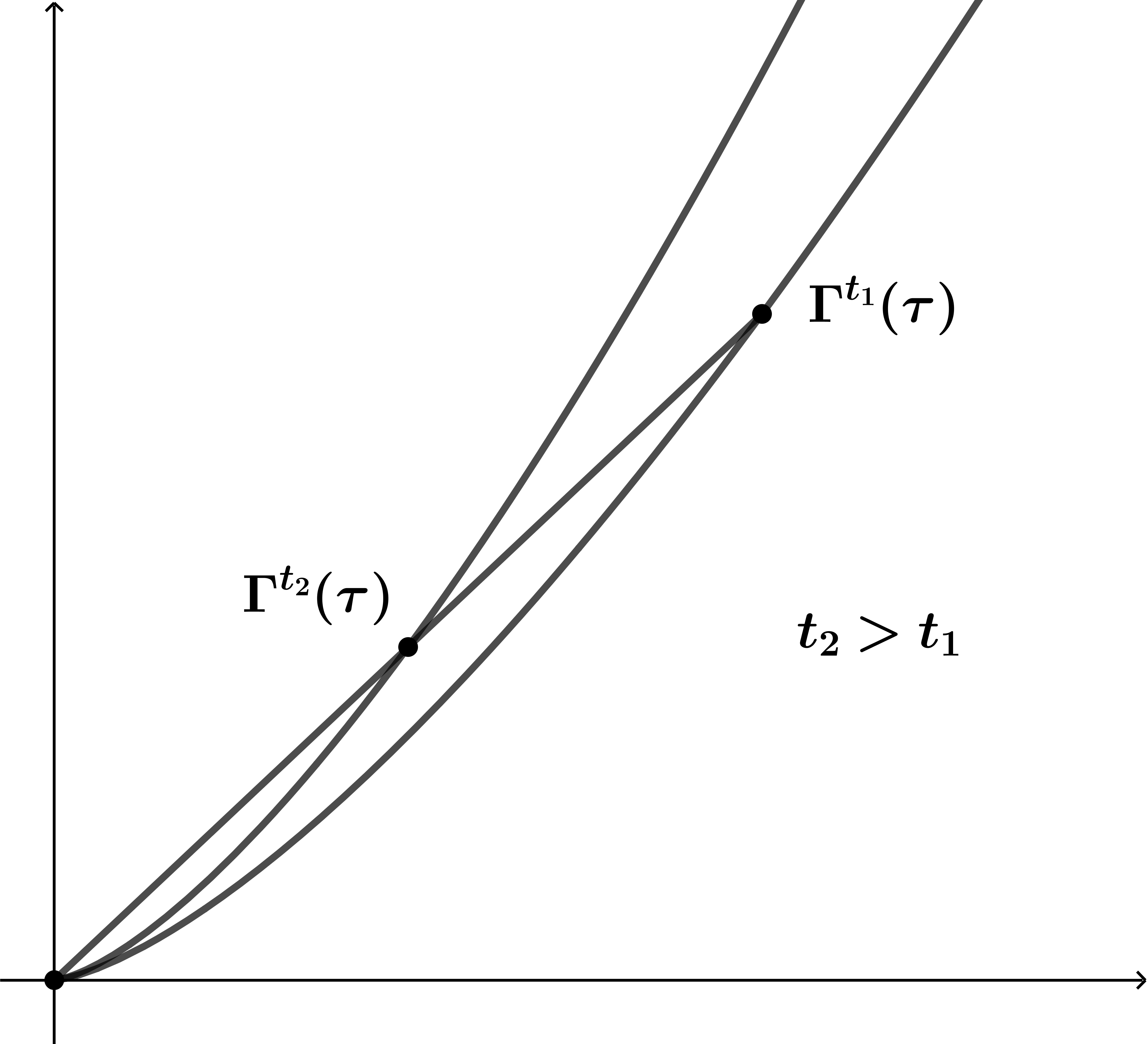}
    \caption{$\gmt{t_1}$ vs. $\gmt{t_2}$}
    \label{fig2}
\end{figure}
\begin{lemma}\label{lem3}
For any $t>0$ the function 
$$
\tau \mapsto \gmt{t}_2(\tau)  - \big(\gmt{t}_1(\tau)\big)^2 
$$
is strictly increasing on $[0,t]$ and strictly decreasing on $[t,\infty).$ Its maximum, attained at $\tau=t,$ equals $\xi^2(t).$

In addition,
\eq[eq3]{
t s (\rr^2(s) + \xi^2(s)) \leq s^2 \rr^2(s) + t^2\xi^2(t), \qquad t,s>0,
}
where equality holds only for $s=t.$
\end{lemma}
\begin{proof}
We use~\eqref{eq1} and~\eqref{eq2}:
$$
\frac{d}{d\tau } \left(\gmt{t}_2(\tau)  - \big(\gmt{t}_1(\tau)\big)^2 \right) = 
\frac{1}{t^2} \AAA''(\tau) (t - \tau).
$$
The last expression is positive for $\tau \in (0,t)$ and is negative for $\tau > t$. Direct computation shows that $\gmt{t}_2(t)  - \big(\gmt{t}_1(t)\big)^2=\xi^2(t).$ Lastly, inequality~\eqref{eq3} is equivalent to the inequality
$
t^2\gmt{t}_2(s)\le t^2\big[(\gmt{t}_1(s))^2+\xi^2(t)\big].
$
\end{proof}
We have an immediate corollary.
\begin{corollary}\label{cor1}
For any $t>0$ the curve $\gmt{t}$ lies below the parabola $x_2 = x_1^2 +\xi^2(t);$ the two are tangent when $x_1=\rr(t),$ which is their only point in common. 
Consequently, the curve $\gt{t}$ intersects each parabola of the form $x_2=x_1^2 + C^2,$ $C \in [0,\xi(t)],$ once; see Figure~\ref{fig3}.
\end{corollary}

\begin{lemma}
\label{l66}
For any $t>0$ and any point $(y_1,y_2) \in \Omega^t$ there exist unique $\tau = \ttt(y_1,y_2, t) \in [0,t]$ and $u = \uuu(y_1,y_2,t) \in \mathbb{R}$ such that 
\begin{gather}
    y_1 = u +  \gt{t}_1(\tau), \label{eq10}\\
    y_2 = u^2 + 2u\gt{t}_1(\tau)  + \gt{t}_2(\tau). \label{eq11}
\end{gather}
The functions $\uuu$ and $\ttt$ satisfy the following homogeneity relations: for any $c\in\mathbb{R}$
\begin{gather}
\label{hom1}\ttt(y_1+c,y_2+2cy_1+y_1^2,t)=\ttt(y_1,y_2,t),\\
\label{hom2}\uuu(y_1+c,y_2+2cy_1+y_1^2,t)=\uuu(y_1,y_2,t)+c.
\end{gather}
\end{lemma}

\begin{remark}
The geometric meaning of the functions $\ttt$ and $\uuu$ is explained in Figure~\ref{fig3}.
\begin{figure}[h!]
    \centering
    \includegraphics[width=0.5\textwidth]{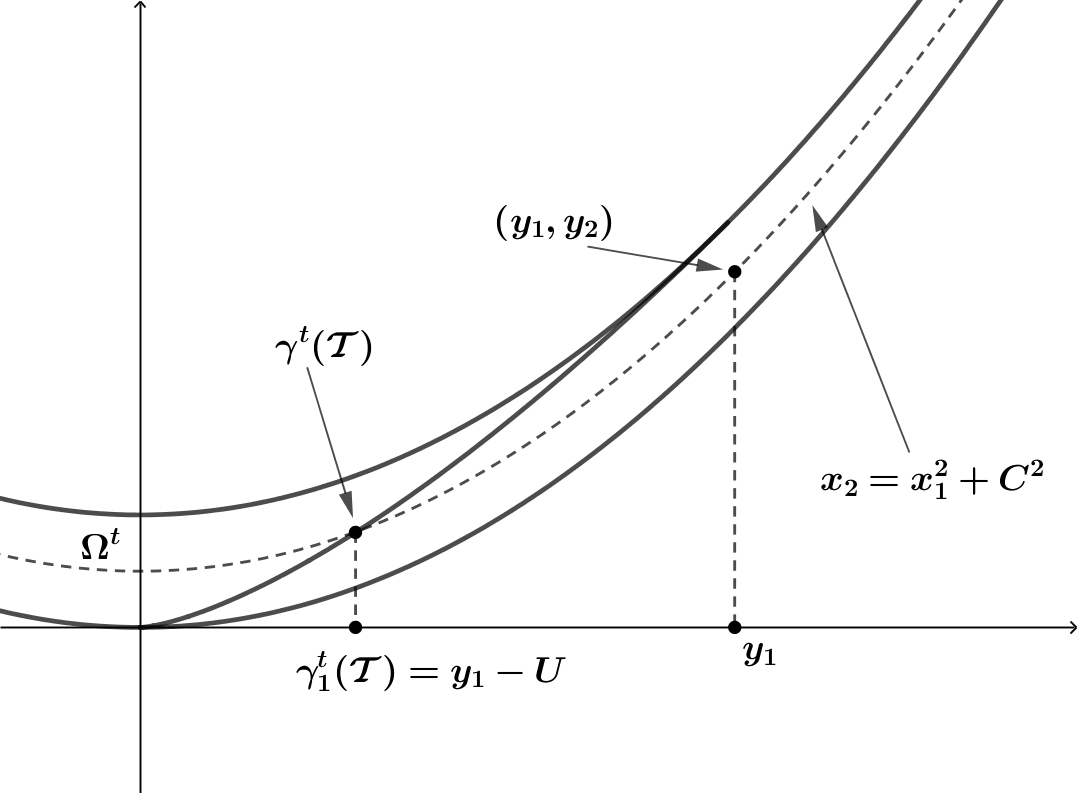}
    \caption{Definition of $\uuu$ and $\ttt$}
    \label{fig3}
\end{figure}
\end{remark}

\begin{proof}[Proof of Lemma~\ref{l66}]
There is a unique parabola of the form $x_2=x_1^2+C^2$ passing through the point $(y_1,y_2)$: we take $C^2=y_2-y_1^2.$ By Corollary~\ref{cor1}, that parabola intersects the curve $\gt{t}$ at a single point. That point is precisely $\gt{t}(\ttt)$ and we thus have
$$
y_2-y_1^2=\gt{t}_2(\ttt)-\big(\gt{t}_1(\ttt)\big)^2
$$
and 
$$
\uuu=y_1-\gt{t}_1(\ttt).
$$
The homogeneity relations~\eqref{hom1} and \eqref{hom2} are elementary.
\end{proof}
For future use, we record several useful statements. The first two are obvious. The last follows from Lemma~\ref{lem2}; see Figure~\ref{fig4}.
\begin{corollary}
\label{rem66}
The functions $\uuu$ and $\ttt$ satisfy the following relations:
\begin{itemize}
    \item $\uuu \in [y_1- \rr(t), y_1].$
    \item If $y_2=y_1^2,$ then $\ttt = 0$ and $\uuu = y_1.$
    \item For $0<t_1<t_2$ and any $y\in\Omega^{t_1}$ such that $y_2>y_1^2$ we have 
    \eq[eq32]{
\uuu(y,t_1) < \uuu(y,t_2).}
\end{itemize}
\end{corollary}

\begin{figure}
    \centering
    \includegraphics[width=0.6\textwidth]{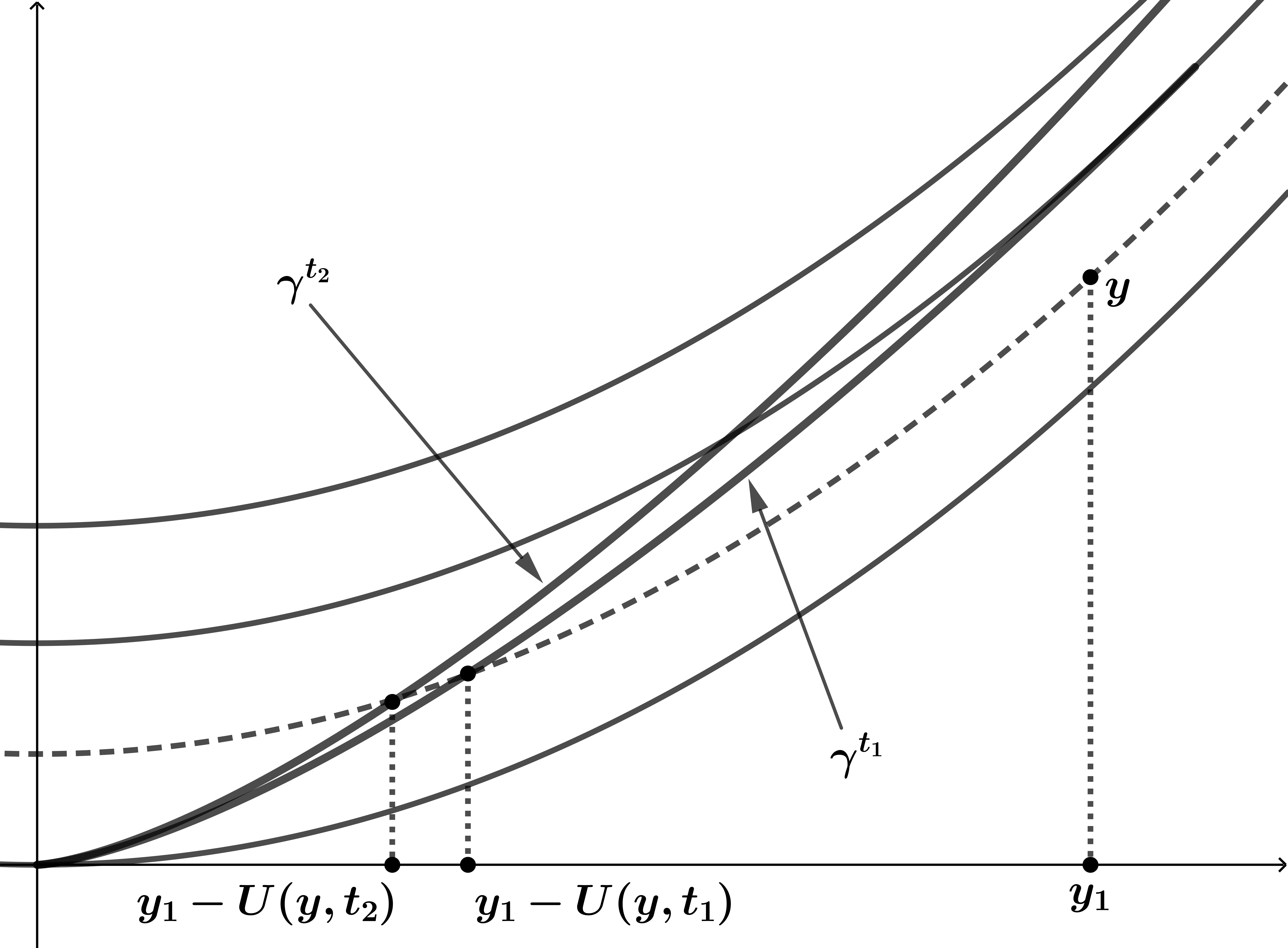}
    \caption{The monotonicity of $U$ with respect to $t$}
    \label{fig4}
\end{figure}

\subsection{An upper bound for $\inf_I\phi$} For $a \in \mathbb{R}$ define 
$$
\Omega^t_{a} = \{x \in \Omega^t \colon \uuu(x,t) \geq a\}.
$$
The following theorem is a generalization of Theorem~3.3 from~\cite{Os} (where it was proved for $\xi(t)=t^\alpha$) with a somewhat different proof and without an additional assertion of sharpness that we do not need here.
\begin{theorem}\label{thAdam}
Let $\xi$ be a modulus on $[0,|I|]$ satisfying conditions~\eqref{611}. If $\phi \in \Xicl{\xi}(I)$ and $\|\phi\|_{\Xicl{\xi}(I)}\leq 1$, then
\eq[eq24]{
\inf_{I}\phi \leq \uuu(\av{\phi}{I},\av{\phi^2}{I},|I|).
}
Equivalently, the point $(\av{\phi}{I},\av{\phi^2}{I})$ lies in $\Omega_{\inf_{I}\phi}^{|I|}$.
\end{theorem}
The key ingredient in the proof of this theorem is the following result.
\begin{lemma}\label{lemAdam}
If $\phi \in \Xicl{\xi}(I)$ and $\|\phi\|_{\Xicl{\xi}(I)}< 1$, then there exists a subinterval $J \subset I$, $0<|J|<|I|$, such that 
$$\uuu(\av{\phi}{J},\av{\phi^2}{J},|J|) \leq \uuu(\av{\phi}{I},\av{\phi^2}{I},|I|).$$
\end{lemma}
\begin{proof}

Let $x_1 = \av{\phi}{I}$, $x_2 = \av{\phi^2}{I}$, and $t=|I|$. Without loss of generality, we can take $I=[0,t]$. By adding a constant to $\varphi$ if necessary, we may also assume $\uuu(x_1,x_2,t)=0$. Then,
$$
x_1 = \gmt{t}_1(\tau),\qquad x_2 = \gmt{t}_2(\tau),
$$
for $\tau = \ttt(x_1,x_2,t) \in [0,t]$. 

If $x_2 = x_1^2$, then $\phi = x_1$ a.\,e. on $I$ and $x_1  = 0.$ Therefore, for any subinterval $J \subset I$ we have $\uuu(\av{\phi}{J},\av{\phi^2}{J},|J|) = 0$.
In what follows, we assume $x_2>x_1^2$ and, thus, $\tau >0$. Since $\|\phi\|_{\Xicl{\xi}(I)}< 1$ we have $x_2 - x_1^2 < \xi^2(t)$, which means that $\tau<t$ and $x_1 < \rr(t)$. By continuity, there exists a $t_0 \in (\tau,t)$ such that $\frac{t}{t_0}x_1 < \rr(t_0)$.

Let $J\subset I$ with $0<|J|<|I|.$ If $\av{\phi}{J}\leq 0$, then $\uuu(\av{\phi}{J},\av{\phi^2}{J},|J|)\leq \av{\phi}{J}\leq 0$, and the claim is proved. Therefore, in what follows we assume $\av{\phi}{J}> 0$ for any such $J$ and, in particular, that $\phi \geq 0$ a.\,e. on $I$. Let 
$$
K(J) = \frac{\av{\phi^2}{J}}{\av{\phi}{J}}, \qquad J \subset I, \quad |J|>0.
$$

For $s \in (0,t)$ consider the intervals $J_-(s)  = (0,s)$ and $J_+(s) = (s,t)$. Note that $K(I)$ lies between $K(J_-(s))$ and $K(J_+(s)),$ and that $K(I) = K(J_-(s))$ if and only if $K(I) = K(J_+(s))$.

Suppose that for some $s \in (0,t)$ we have  
\eq[eq8]{
K(J_-(s)) = K(I) = K(J_+(s)).
}
Define $J = J_-$ or $J = J_+$ in such a way that $\av{\phi}{J}\leq x_1$. Take $\lambda = \frac{x_1}{\av{\phi}{J}}$, $\lambda \geq 1$. Then,
$$
\av{\phi}{J} = \frac{x_1}{\lambda} = \frac{\gmm_1(\tau)}{\lambda t} =  \gmt{\lambda t}_1(\tau),
\qquad
\av{\phi^2}{J} = \frac{x_2}{\lambda} = \frac{\gmm_2(\tau)}{\lambda t} = \gmt{\lambda t}_2(\tau).  
$$
We note that $\tau \in [0, \lambda t]$; therefore, $\ttt(\av{\phi}{J}, \av{\phi^2}{J}, \lambda t) = \tau$ and
$$
0 = \uuu(\av{\phi}{J}, \av{\phi^2}{J}, \lambda t) \geq \uuu(\av{\phi}{J}, \av{\phi^2}{J}, |J|), 
$$
where the last inequality follows from~\eqref{eq32}. The claim is proved. 

We are left to consider the case when~\eqref{eq8} does not hold for any $s \in (0,1)$. 
By continuity of $K(J_\pm(s))$ with respect to $s$, we have either
$$
K(J_-(s)) < K(I) < K(J_+(s)) \qquad \text{ or } \qquad K(J_+(s)) < K(I) < K(J_-(s)) 
$$
for all $s \in (0,t)$ simultaneously. Then, $K(I)$ lies between $K(J_-(t_0))$ and $K(J_+(t-t_0))$. The function $s \mapsto K((s, s+t_0))$ is continuous on $[0,t-t_0]$; hence, we can find  $s \in [0,t-t_0]$ such that for $J = (s,s+t_0)$ we have $K(J)= K(I)$. Since $\phi \geq 0$, we have
 $$
 \av{\phi}{J} \leq \frac{|I|}{|J|}\av{\phi}{I} = \frac{t}{t_0}x_1 < r(t_0).
 $$
Take $\lambda  = \frac{\gmm_1(\tau)}{t_0 \av{\phi}{J}}=\frac{t x_1}{t_0\av{\phi}J},$ $\lambda\geq 1$. Then 
$$
 \av{\phi}{J} =  \frac{\gmm_1(\tau)}{\lambda t_0} = \gmt{\lambda t_0}_1(\tau), \qquad 
 \av{\phi^2}{J} = K(J)\av{\phi}{J} = K(I) \av{\phi}{J} = \frac{\av{\phi^2}{I} }{\av{\phi}{I} } \frac{\gmm_1(\tau)}{\lambda t_0} = \frac{\gmm_2(\tau)}{\lambda t_0} = \gmt{\lambda t_0}_2(\tau).
$$
Note that $\tau \in [0, \lambda t_0]$; therefore, $\ttt(\av{\phi}{J}, \av{\phi^2}{J}, \lambda t_0) = \tau$ and
$$
0 = \uuu(\av{\phi}{J}, \av{\phi^2}{J}, \lambda t_0) \geq \uuu(\av{\phi}{J}, \av{\phi^2}{J}, |J|), 
$$
where the last inequality follows from~\eqref{eq32}. This completes the proof. 
\end{proof}

\newcommand{\uf}{L}

\begin{proof}[Proof of Theorem~\ref{thAdam}]
For any subinterval $J \subset I$ let 
$$
\uf(J) = \uuu(\av{\phi}{J},\av{\phi^2}{J},|J|).
$$

It suffices to prove the statement of the theorem for functions $\phi$ satisfying $\|\phi\|_{\Xicl{\xi}(I)}< 1$ and such that $\uf(I)=0$. Let
$$
\uf = \inf\{\uf(J) \colon J \subset I\},
$$
where the infimum is taken over all subintervals $J$. Obviously, $\uf \leq \uf(I)= 0$. Find a sequence of intervals $J_n$ such that $\uf(J_n) \to \uf$ and $D = \lim |J_n|$ is minimal. If $D > 0$, then we can find an interval $J\subset I$, $|J|=D$, and a subsequence $J_{n_k}$ that converges to $J$. Then, 
$$
\uf(J) = \lim \uf(J_{n_k}) = \uf
$$
by the continuity of the function $\uuu.$
Use Lemma~\ref{lemAdam} to find an interval $\tilde J \subset J$ such that $\uf(\tilde J)\leq \uf(J)$ and $0<|\tilde J|< |J| = D$. This contradicts the minimality of $D$. Thus, $D=0$ and, therefore, for any $\delta>0$ we can find an interval $J \subset I$ such that $\uf(J) \leq \delta$ (since $L\le0$) and $\rr(|J|)<\delta$ (see Lemma~\ref{rem1}). Then, 
$$
\inf_{I}\phi \leq \inf_J \phi \leq \av{\phi}{J} \leq \uf(J) + \rr(|J|) \leq 2\delta.
$$
where the middle inequality follows from the first statement in Corollary~\ref{rem66}. We conclude that $\inf_{I}\phi \leq 0$. The theorem is proved.
\end{proof}

\subsection{The cutout function and the end of the proof}
Let $I = [t_1,t_2] \subset \mathbb{R}$ be an interval and let $E \subset \mathbb{R}$ be a set of positive measure such that $d:= |I\setminus E|>0$. Define a function $h_{I,E} \colon I \to \mathbb{R}$ by
$$
h_{I,E}(s) = \big|[t_1,s] \setminus E\big|.
$$
This function is non-decreasing and Lipshitz; its image is the interval $[0, d]$. For all but at most countable set of $\tau \in [0,d]$ the preimage $h_{I,E}^{-1}(\tau)$ is a one-point set. 

For a function $\phi\colon I \to \mathbb{R}$ define the {\it cutout of $\phi$ along $E$} to be the function $\col_{I,E}\phi$ given by: 
$$
\col_{I,E}\phi(\tau) =  \phi\big(h_{I,E}^{-1}(\tau)\big), \qquad \tau \in [0,d].
$$
Note that this function is correctly defined for all but a countable set of $\tau \in [0,d]$, and is measurable. The construction of $\col_{I,E}$ is illustrated in Figure~\ref{cutout}.
\begin{figure}[h!]
\begin{center}
    \includegraphics[width=0.4\textwidth]{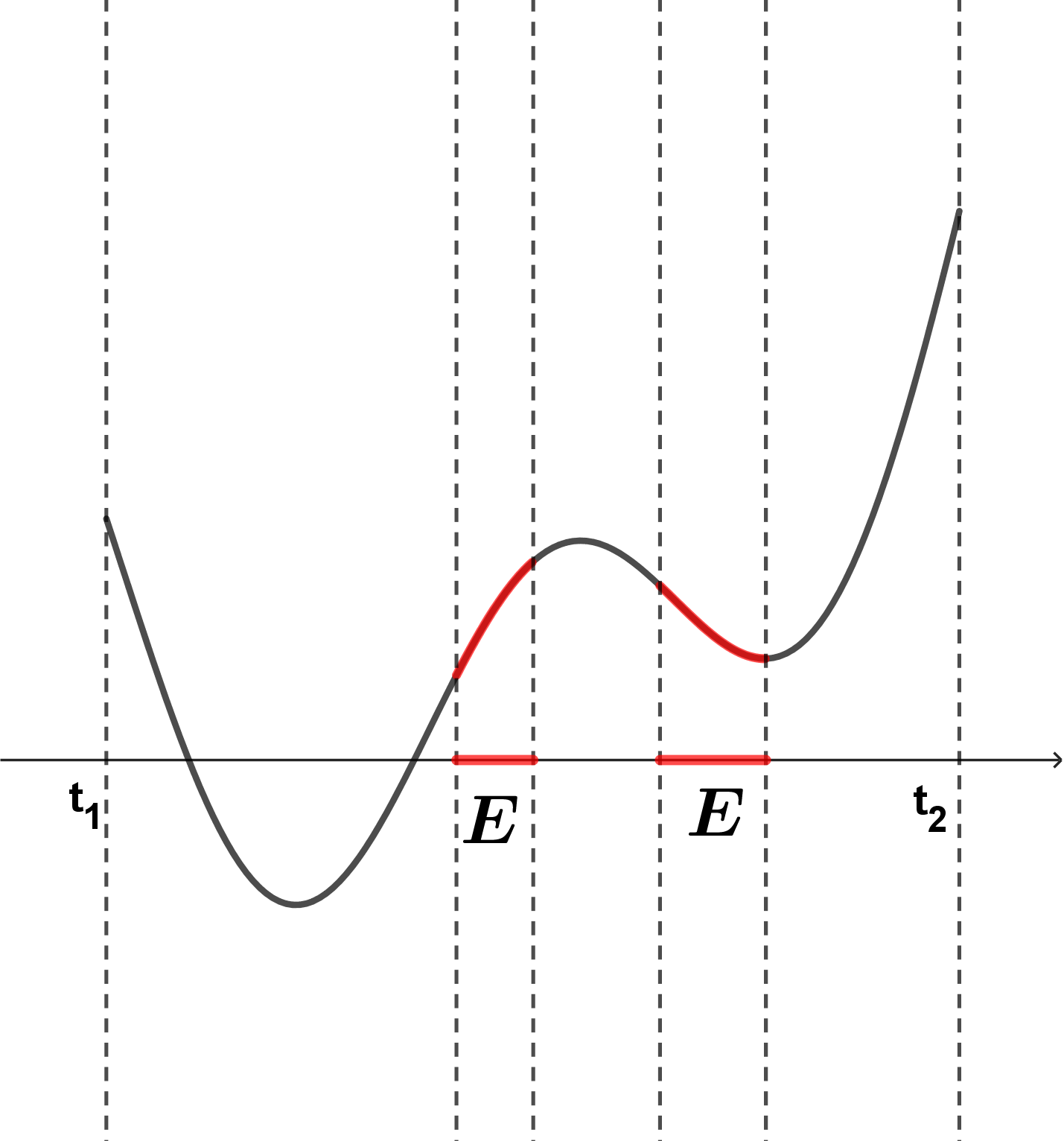}
    \includegraphics[width=0.4\textwidth]{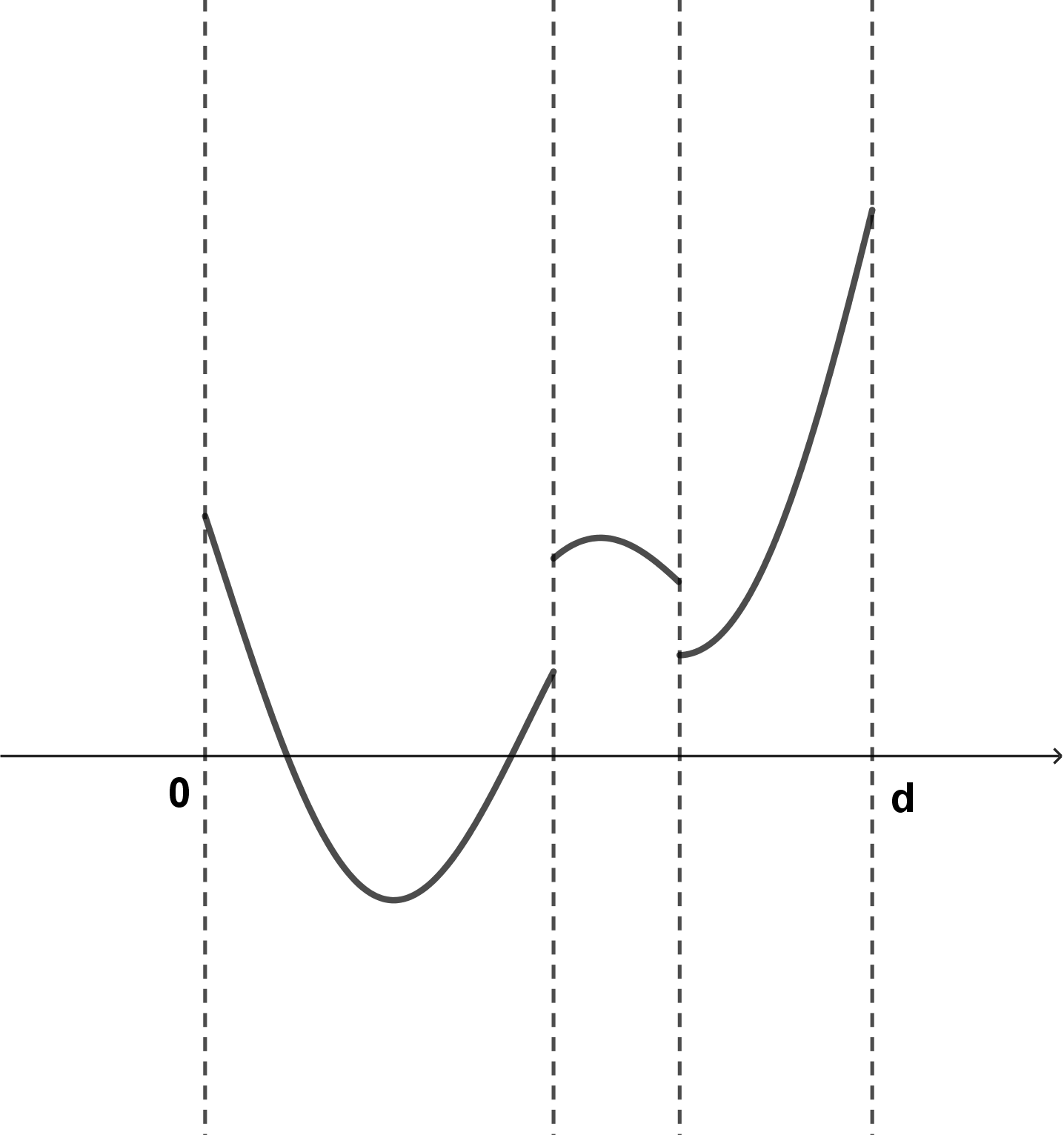}
    \caption{The cutout function}
    \label{cutout}
\end{center}    
\end{figure}

\begin{prop}\label{lem4}
Let $\phi \in \Xicl{\xi}(I)$ be such that $\|\phi\|_{\Xicl{\xi}(I)}\leq 1$ and $\inf_I\phi \geq a$ for some  $a\in \mathbb{R}.$ Let $E \subset I$ be such that $\phi = a$ a.\,e. on $E.$ Let $d=|I\setminus E|$. Then, $\psi = \col_{I,E}\phi$ lies in $\Xicl{\xi}([0,d])$ and $\|\psi\|_{\Xicl{\xi}([0,d])}\leq 1$.
\end{prop}
\begin{remark}
We note that the same principle also holds for BMO: the cutout of a BMO function $\phi$ along the set where it attains its minimum does not increase its norm. Indeed, both Proposition~\ref{lem4} and its analog for BMO follow from the fact that the rearrangement does not increase the relevant norm (Theorem~\ref{th1} in the present case; Corollary~3.12 from~\cite{sz} for BMO and other averaging classes). In turn, as seen here, this ``cutout principle'' is itself the key to proving the rearrangement theorem.
\end{remark}
To prove Proposition~\ref{lem4}, we first need another geometric fact about the domain $\Omega^t_a.$ 
\begin{lemma}
\label{lem55}
Take $a \in \mathbb{R}.$ For $x \in \Omega^t_a$ and $s\in (0, t)$, let 
\eq[eq25]{
y = (y_1,y_2) =  (a,a^2) + \frac{t}{s} \big(x-(a,a^2)\big).
}
Then, 
$y_2 \leq y_1^2 + \xi^2(s).$
\end{lemma}
\begin{proof}
Note that the statement of the lemma is invariant under the parabolic shift $(x_1,x_2)\mapsto (x_1-a,x_2-2ax_1+a^2).$ Therefore, we can assume that $a=0.$

We consider two cases: $x_1 \in[0, \rr(t)]$ and $x_1>\rr(t)$. 
If $x_1 \in[0, \rr(t)]$, then the point $(x_1,x_2)$ lies below the curve $\gt{t}$; thus, there exists $\tau \in [0,t]$ such that 
$$
x_1 = \gmt{t}_1(\tau), \qquad x_2 \leq \gmt{t}_2(\tau).
$$
Then, 
$$
y_1 = \frac{t}{s}\gmt{t}_1(\tau) = \gmt{s}_1(\tau), \qquad y_2 \leq \frac{t}{s}\gmt{t}_2(\tau) = \gmt{s}_2(\tau),
$$
which means that the point $y$ lies below the curve $\gmt{s}$. The claim now follows from Corollary~\ref{cor1}.

Assume now that $x_1>\rr(t).$ For this case, the argument is different: here we only use the fact that $x_2 \leq x_1^2 + \xi^2(t)$. Since 
$$
y_2 = \frac{t}{s}x_2 \leq \frac{t}{s}(x_1^2 + \xi^2(t)), 
$$
it suffices to prove that
$$
\frac{t}{s}(x_1^2 + \xi^2(t)) \leq \frac{t^2}{s^2}x_1^2 + \xi^2(s).
$$
The last inequality is equivalent to 
\eq[eq4]{
ts\xi^2(t) \leq (t^2-ts)x_1^2 + s^2\xi^2(s).
}
Interchanging $s$ and $t$ in inequality~\eqref{eq3} of Lemma~\ref{lem3}, we have
$$
ts \xi^2(t) \leq (t^2-ts)\rr^2(t) + s^2\xi^2(s),
$$
which implies~\eqref{eq4} since $x_1 > \rr(t)$ and $t^2>ts$.
\end{proof}

\begin{proof}[Proof of Proposition~\ref{lem4}]
Let $\tau_1, \tau_2 \in [0, d]$, $\tau_1<\tau_2$, $J = [\tau_1,\tau_2]$, $s = |J|$. We need to verify that the point
$$
y = (y_1,y_2) = (\av{\psi}{J}, \av{\psi^2}{J})
$$
lies in the strip $\Omega^s$. Without loss of generality, we can assume that $h_{I,E}^{-1}(\tau_1)$ and $h_{I,E}^{-1}(\tau_2)$ are one-point sets. Let $\tilde J = h_{I,E}^{-1}(J)$, $t = |\tilde J|$, and
$$
x = (x_1,x_2) = (\av{\phi}{\tilde J}, \av{\phi^2}{\tilde J}).
$$
Theorem~\ref{thAdam} says that $\uuu(x,t)\ge\inf_{\tilde{J}}\phi\ge a,$ i.\,e., $x \in \Omega_a^t$. It is easy to see that equality~\eqref{eq25} holds. Therefore, by Lemma~\ref{lem55}, $y_2 \leq y_1^2+ \xi^2(s).$ Since $y_2\ge y_1^2,$ we have $y \in \Omega^s$.
\end{proof}

Surely, the symmetric statement is also true:
\begin{prop}\label{lem5}
Let $\phi \in \Xicl{\xi}(I)$ be such that $\|\phi\|_{\Xicl{\xi}(I)}\leq 1$ and $\sup_I\phi \leq b$ for some  $b\in \mathbb{R}.$ Let $E \subset I$ be such that $\phi = b$ a.\,e. on $E.$ Let $d=|I\setminus E|$. Then, $\psi = \col_{I,E}\phi$ lies in $\Xicl{\xi}([0,d])$ and $\|\psi\|_{\Xicl{\xi}([0,d])}\leq 1$.
\end{prop}

Combining Proposition~\ref{lem4} and Proposition~\ref{lem5} yields the following 
\begin{prop}\label{prop1}
Let $\phi \in \Xicl{\xi}(I)$ be such that $\|\phi\|_{\Xicl{\xi}(I)}\leq 1$ and $a \leq \phi \leq b$ on $I$ for some $a, b \in \mathbb{R}.$ Let $E_a \subset I$ and $E_b\subset I$ be such that $\phi = a$ a.\,e. on $E_a$ and $\phi = b$ a.\,e. on $E_b.$ Write $E = E_a \cup E_b$ and let $d=|I\setminus E|$. Then, $\psi := \col_{I,E}\phi$ lies in $\Xicl{\xi}([0,d])$ and $\|\psi\|_{\Xicl{\xi}([0,d])}\leq 1$.
\end{prop}
We are now in a position to prove Theorem~\ref{th1} subject to the assumptions~\eqref{611}.
\begin{proof}[Proof of Theorem~\ref{th1}]
Without loss of generality, assume that $I=[0,t]$ and $\|\phi\|_{\Xicl{\xi}(I)}= 1$.

Take any $t_1,t_2$ such that $0<t_1<t_2<t$. Let $b = \phi^*(t_1)$, $a = \phi^*(t_2)$, $b \geq a$. We need to prove that 
\eq[eq26]{
\av{(\phi^*)^2}{[t_1,t_2]} - \av{\phi^*}{[t_1,t_2]}^2 \leq \xi^2(t_2-t_1).
}
We can consider the truncation of $\phi,$ and thus of $\phi^*,$ at the levels $a$ and $b$ in place of the original function. Since on any interval $J\subset I$ we have $\av{\phi^2}J-\av{\phi}J^2=\frac1{|J|^2}\int_J\int_J(\phi(x)-\phi(y))^2dxdy,$
such truncation does not increase the norm of $\phi$ and does not affect~\eqref{eq26}. 

Find sets $E_a \subset I$, $E_b \subset I$ such that $\phi = a$ on $E_a$, $\phi = b $ on $E_b$, $|E_a| = t-t_2$, and $|E_b| =t_1$. Let $E = E_a \cup E_b$ and $d=|I\setminus E| = t_2-t_1$. From Proposition~\ref{prop1} we know that the function $\psi = \col_{I,E}\phi$ lies in $\Xicl{\xi}([0,d])$ and its $\Xicl{\xi}([0,d])$-norm is bounded by $1$. But the distribution of $\phi^*$ on $[t_1,t_2]$ coincides with the distribution of $\psi$ on $[0,d]$, which proves~\eqref{eq26}.
\end{proof}

\section{Proof of Theorem~\ref{th1} in the non-smooth case}
\label{non-smooth}
In the case when the modulus $\xi$ does not satisfy the additional smoothness assumptions of Section~\ref{smooth}, we employ a mollification procedure. Note, that for now we still assume condition~\eqref{eqConv} on~$\xi$; that will be removed in Section~\ref{SecNew}.

Take a function $\ck \in C_0^\infty(0,\infty)$ such that $\ck \geq 0$, $\int u \ck(u)\,du = 1,$ and $\int \ck(u) \,du=1$. For a continuous function $F$ defined on $(0,\infty)$ consider the following multiplicative convolution:
\begin{align}
\label{eq221101}
F\con \ck(t) :=& \int_0^\infty \frac{F(v)}{t}\,\ck\Big(\frac{v}{t}\Big) \, dv =\\ 
=& \int_0^\infty F(tu) \ck(u) \, du, \qquad t>0.  
\label{eq221102}
\end{align}
From~\eqref{eq221101} we see that $F\con \ck$ is a smooth function. From~\eqref{eq221102} it is clear that $F\con \ck$  inherits monotonicity and convexity properties of $F$. Moreover, if $F$ is convex, applying Jensen's inequality with the measure $\ck(u)\,du$ we conclude that
\begin{equation}\label{eq35}
F\con \ck(t) = \int_0^\infty F(tu) \ck(u) \, du \geq F\Big(\int_0^\infty  tu \ck(u) \, du\Big) = F(t).
\end{equation}

Now, take a sequence $\ck_n$  whose supports converge to the set $\{1\}$. Then $F\con \ck_n \to F$ pointwise on $(0,\infty)$.

Suppose that a modulus $\xi$ satisfies~\eqref{eqConv}.  
 Define a sequence of functions $\{\xi_n\}$ by
$$
\xi_n\colon t \mapsto \frac{1}{t}\Big(A \con \ck_n(t)+\frac{t^3}{n}\,\Big)^{1/2}, \qquad t>0.
$$ 
\begin{lemma}\label{lem7}
The functions $\xi_n$ have the following properties:
\begin{enumerate}
    \item $\xi_n$ is smooth, $\xi_n>0,$ and $\xi_n'>0$ on $(0,\infty)$;
    \item $\xi_n(t) \to 0$ as $t \to 0^+$;
    \item $A_n\colon t \mapsto t^2\xi_n^2(t)$ is convex on $(0,\infty)$, $A_n''(t)>0$ for $t>0$;
    \item $\xi_n(t) \geq \xi(t)$ for $t>0$;
    \item $\xi_n(t) \to \xi(t)$ for $t>0$ as $n\to \infty$.
\end{enumerate}
\end{lemma}
\begin{proof}
(1) To prove that $\xi_n$ is increasing we note that 
\eq[eq36]{
\xi_n^2(t) = \frac{t}{n} + \frac{1}{t^2} \, A \con \ck_n(t)
= \frac{t}{n} + \int_0^\infty \frac{A(tu)}{t^2} \ck_n(u) \, du=
\frac{t}{n} + \int_0^\infty \xi^2(tu) u^2 \ck_n(u) \, du,
}
which is an increasing function with respect to $t$ because $\xi$ is increasing; moreover, $\xi_n'(t)>0$.

(2) Since $\Psi_n$ has compact support in $(0,\infty)$ and $\lim_{t \to 0^+}\xi(t) =0$, from~\eqref{eq36} we see that $\lim_{t \to 0^+}\xi_n(t) =0$.

(3) We have
\eq[eq37]{
A_n(t) = t^2\xi_n^2(t) = \frac{t^3}{n} + A \con \ck_n(t).
}
The second summand here is convex because $A$ is convex, while the first one has a positive second derivative. Therefore, $A_n''(t)>0$.

(4) $\xi_n(t) \geq \xi(t)$ is equivalent to  
$A_n(t)\geq A(t)$, which follows from~\eqref{eq37} and~\eqref{eq35}.

(5) Letting $n\to\infty$ in~\eqref{eq37} we get $\lim_{n\to \infty} A_n(t) = A(t)$ for any $t>0$.
\end{proof}

The proof of Theorem~\ref{th1} in the non-smooth case with the assumption~\eqref{eqConv} is now immediate.
Fix a function $\phi \in \Xicl{\xi}(I)$. Since $\xi_n \geq \xi$, we have the estimate
$\|\phi\|_{\Xicl{\xi_n}}\leq \|\phi\|_{\Xicl{\xi}}$. Then, by Theorem~\ref{th1} in the smooth case, for the monotone rearrangement $\phi^*$ we have $\|\phi^*\|_{\Xicl{\xi_n}}\leq \|\phi\|_{\Xicl{\xi_n}}\leq \|\phi\|_{\Xicl{\xi}}$. Thus, for any subinterval $J\subset I$ 
we have the following estimate
$$
\av{(\phi^*)^2}{J} - \av{\phi^*}{J}^2 \leq \|\phi\|_{\Xicl{\xi}}^2 \xi_n^2(|J|).
$$
Letting $n\to\infty$ we conclude that $\phi^*$ satisfies~\eqref{eqmain}.

\section{Proof of Theorem~\ref{th1} in the general case}\label{SecNew}
In this section, we present an argument that allows us to dispose of the convexity assumption~\eqref{eqConv} on the modulus $\xi.$ This argument was suggested to us by Fedor Nazarov. His idea was roughly this: given an arbitrary modulus $\xi$ and a point $t_0$, find a modulus satisfying~\eqref{eqConv}, majorating $\xi,$ and coinciding with it at $t_0.$ An application of Theorem~\ref{th1} then yields the desired estimate for oscillations of $\phi^*$ over intervals of length $t_0.$
We have found it necessary to implement a sequence of majorants, but the main idea is the same.

\begin{lemma}\label{FN}
    Let $\xi$ be a modulus on $[0,\TT]$ and take $t_0 \in (0,\TT]$. Then for any $\delta>0$ there exists a continuous increasing function $\txi$ on $[0,\TT]$ such that $\txi\geq \xi$ on $[0,\TT]$, $\txi(0)=0$, $\txi(t_0)\leq \xi(t_0)+\delta$, and $t\mapsto t\txi(t)$ is convex on $[0,\TT]$. 
\end{lemma} 
\begin{proof}
    Let $\GG(t) = t \xi(t)$. We will construct a convex increasing function $\tGG$ on $[0,\TT]$ such that $\tGG\geq \GG$, and then define $\txi(t) = \frac{\tGG(t)}{t}$. We will argue graphically.

    First, define $\tGG(t_0) = \GG(t_0)+\delta t_0$. The function $\GG$ is continuous on $[t_0,\TT]$, so we can find a constant $k_0>0$ such that $\tGG(t_0) + k_0(t-t_0) > \GG(t)$ for $t \in [t_0,\TT]$. Define $\tGG(t) = \tGG(t_0) + k_0(t-t_0)$ for $t \in (t_0,\TT]$. Define a point $P_0 = (t_0, \GG(t_0)+\delta t_0)$.

    For $n\geq 0$ draw the rays $R_n = \Big\{(t,\frac{\xi(t_0)}{2^n}t)\colon t>0\Big\}$ from the origin. Starting from $n=1$ we inductively construct decreasing sequences of positive numbers $t_n$ and $k_n$, and points $P_n$, as follows. The part of the graph of $\GG$ in the closed angle with the sides $R_n$ and $R_{n-1}$ is separated from the origin because $\lim_{t\to 0^+}  \frac{\GG(t)}{t} = \lim_{t\to 0^+ }\xi(t) =0$. Thus we can choose a point $P_n$ on the ray $R_n$ in such a way that this part of the graph lies strictly to the right of the segment $P_{n}P_{n-1}$; see Figure~\ref{fig5}. Additionally, we can make the slope of this segment, which we denote by $k_n$, strictly less than $k_{n-1}$, and the first coordinate of $P_n$, which we denote by $t_n$, less than $t_{n-1}/2$. Define $\tGG$ on $[t_{n},t_{n-1}]$ as the affine function whose graph is the segment $P_{n}P_{n-1}$. 
\begin{figure}[h!]
\begin{center}
    \includegraphics[width=0.48\textwidth]{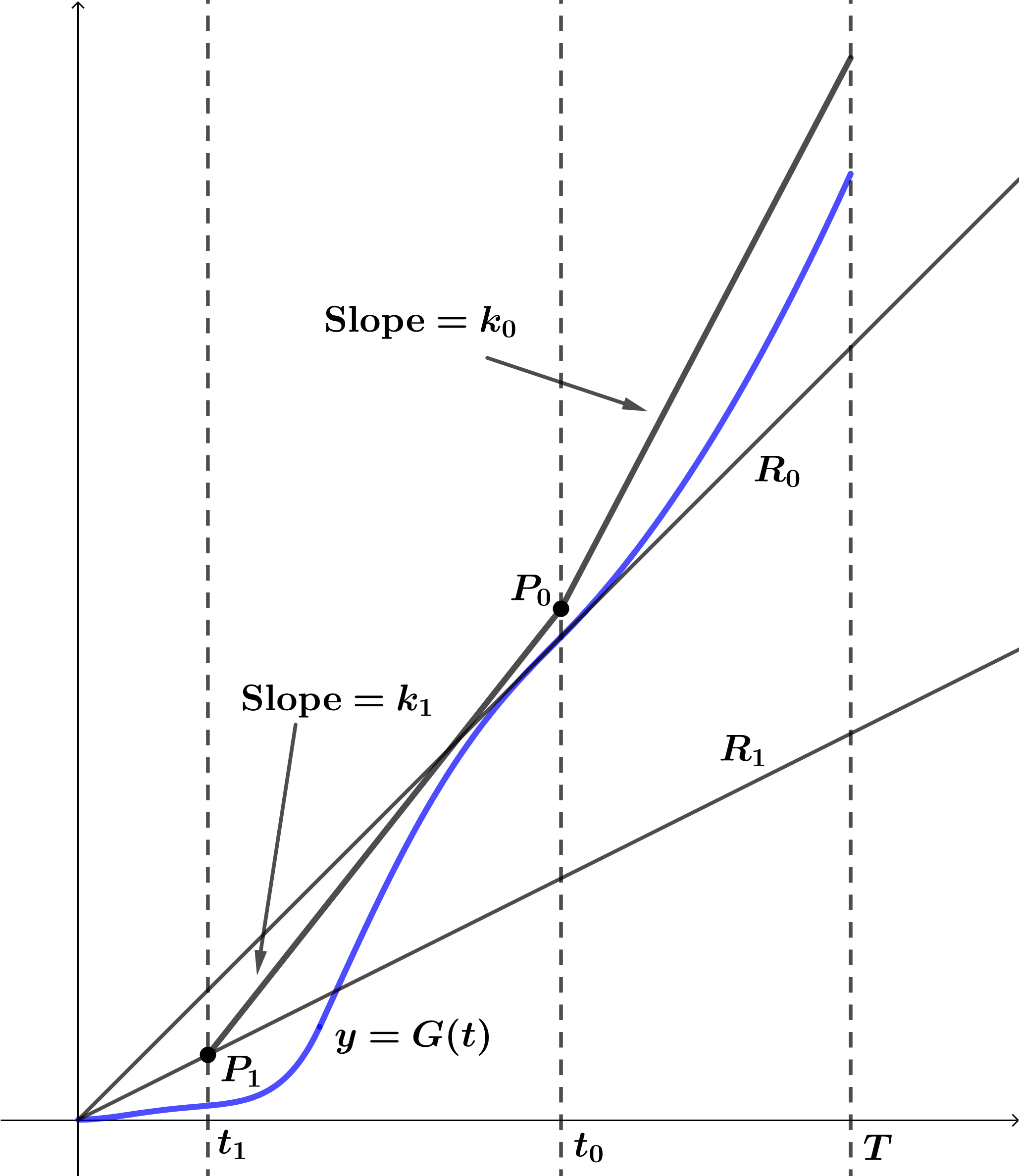}\hfill
    \includegraphics[width=0.48\textwidth]{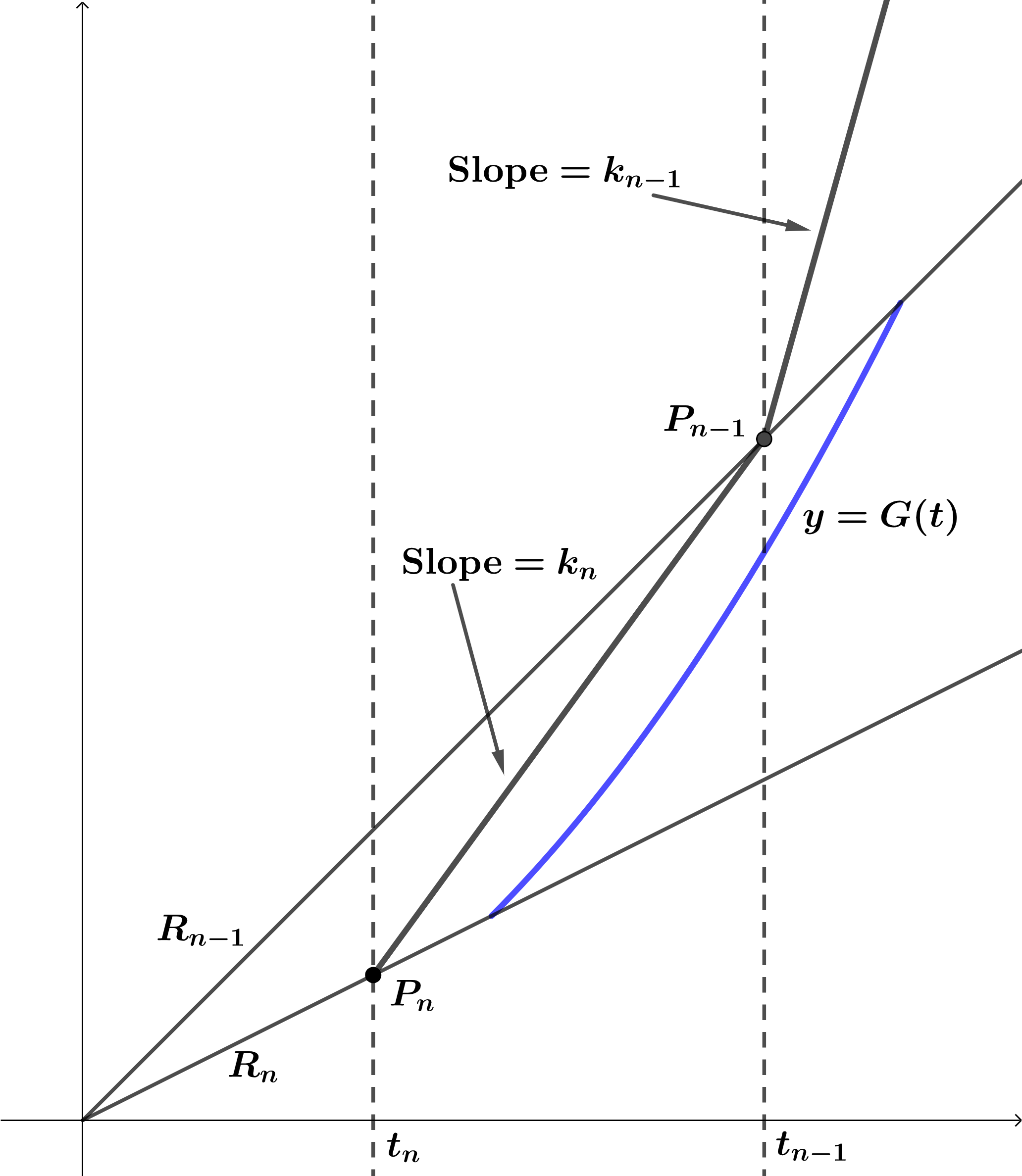}
    \caption{The construction of the points $P_n$}
    \label{fig5}
\end{center}    
\end{figure}

    The function $\tGG$ is now defined on $(0,\TT]$ (because $t_n\to0$), and is convex there because the slopes $k_n$ decrease. By construction, we see that $\tGG \geq \GG$, and that $\lim_{t\to 0^+}\tGG(t)/t =0$. Define $\txi(t) = \tGG(t)/t$ for $t\in (0,\TT]$ and $\txi(0)=0$; we see that this function is as claimed.   
\end{proof}
We are now in a position to prove the statement of Theorem~\ref{th1} in full generality.
\begin{proof}[Proof of Theorem~\ref{th1}]
Write $I=[0,\TT]$ and assume $\|\phi\|_{\Xicl{\xi}(I)}=1.$
We want to show 
\eq[eq12]{
\av{(\phi^*)^2}{J} - \av{\phi^*}{J}^2 \leq \xi^2(|J|)
}
for any subinterval $J\subset I$. It suffices to prove~\eqref{eq12} for a proper subinterval $J$, since for $J=I$ the oscillations of $\phi$ and $\phi^*$ coincide. Let $t_0=|J|$, take any $\delta>0$, and use Lemma~\ref{FN} to find a function $\txi$. Since $\txi \geq \xi$, we have $\|\phi\|_{\Xicl{\txi}(I)}\leq 1$. Since the function $t\mapsto t\txi(t)$ is convex, so is the function $t\mapsto t^2\txi^2(t)$. Thus, we can apply Theorem~\ref{th1} for $\phi\in\Xicl{\txi}(I)$ and obtain $\|\phi^*\|_{\Xicl{\txi}(I)}\leq 1$. Then
\eq[eq13]{
\av{(\phi^*)^2}{J} - \av{\phi^*}{J}^2 \leq 
 \txi^2(t_0)\leq 
(\xi(t_0)+\delta)^2.
}
Tending $\delta$ to zero yields~\eqref{eq12}.
\end{proof}

\section{Proof of Theorem~\ref{th0}, Proposition~\ref{pr00}, and Theorem~\ref{ThCor}}
\label{non-empty}
In this section, we prove all supplementary results stated in the introduction.

\begin{proof}[Proof of Theorem~\ref{th0}]
First, if $\liminf_{t \to 0^+} \frac{\xi(t)}{t}>0$, then there is some $\ve>0$ such that  $\xi(t) \geq \ve t$ for $t \in [0,|I|]$. The function $\psi(s): = \ve s$ satisfies the relation
$$
\av{\psi^2}{J} - \av{\psi}{J}^2 = \ve^2 \frac{|J|^2}{12} < \xi^2(|J|)
$$
for any subinterval $J \subset I$. Therefore, $\psi \in \Xicl{\xi}(I)$.

Conversely, assume that $\liminf_{t \to 0^+} \frac{\xi(t)}{t} =0.$ For any intervals $\tilde J \subset J \subset I$ we have the estimate
$$
|\av{\phi}{\tilde J} - \av{\phi}{J}|^2 \leq \av{|\phi-\av{\phi}{J}|}{\tilde J}^2 \leq
\av{|\phi-\av{\phi}{J}|^2}{\tilde J}\leq 
\frac{|J|}{|\tilde J|}\,\av{|\phi-\av{\phi}{J}|^2}{J}
\leq \frac{|J|}{|\tilde J|}\, \xi^2(|J|),
$$
i.\,e., 
\eq[eq27]{
|\av{\phi}{\tilde J} - \av{\phi}{J}| \leq \sqrt{\frac{|J|}{|\tilde J|}}\, \xi(|J|).}
Let $J_-$ and $J_+$ be the two halves of $J$. Applying~\eqref{eq27} with $\tilde J = J_-$ and $\tilde J=J_+$ we obtain
\eq[eq33]{
|\av{\phi}{J_- } - \av{\phi}{J_+}|\leq
|\av{\phi}{J_- } - \av{\phi}{J}|
+
|\av{\phi}{J_+ } - \av{\phi}{J}|
\leq 2\sqrt{2} \xi(|J|) =  2\sqrt{2} \xi(2|J_-|).}

Fix a subset $\tilde I \subset I$ of full measure such that for $z \in \tilde I$ we have convergence $\av{\phi}{J} \to \phi(z)$ when $J \ni z$,  $|J| \to 0$. For any two different points $z_1,z_2 \in \tilde I$ and for any $d>0$ sufficiently small take an arithmetic progression $a_1, \dots, a_{n+1}$ with difference $d$ such that 
$z_1 \in J_1$ and 
$z_2 \in J_n$, where $J_k:=[a_k,a_{k+1}].$ Then
\eq[eq34]{
(n-2) d  = |a_n-a_2| \leq |z_1-z_2|.
}
Applying \eqref{eq33} repeatedly and then using \eqref{eq34} we obtain
$$
|\av{\phi}{J_1} - \av{\phi}{J_{n}}|\leq \sum_{k=1}^{n-1}|\av{\phi}{J_k}-\av{\phi}{J_{k+1}}|\leq
2\sqrt{2}(n-1) \xi(2d) \leq 2\sqrt{2}\,\frac{|z_1-z_2|+d}{d}\, \xi(2d).
$$

Tending $d$ to zero along an appropriate subsequence, we obtain
$$
|\phi(z_1) - \phi(z_2)|\leq
2\sqrt{2}|z_1-z_2|\, \liminf_{d \to 0^+} \frac{\xi(2d)}{d} =0.
$$
Therefore, $\phi$ is constant almost everywhere on $I$.
\end{proof}

We now turn to Proposition~\ref{pr00} and start with the following lemma. 
\begin{lemma}\label{lem6}
    Let $\psi$  be a monotone function on an interval $[a,b]$, $\psi\in L^2(a,b)$. Then the function $F = F_\psi^{a,b}\colon t\mapsto t \int_a^{a+t}\psi^2 - (\int_a^{a+t}\psi)^2$ is convex on the interval $[0,b-a]$. Similarly, the function
    $t\mapsto t \int_{b-t}^{b}\psi^2 - (\int_{b-t}^{b}\psi)^2$ is convex on $[0,b-a]$.
\end{lemma}
\begin{proof}
    We will only prove the first statement, as the second one is symmetric.
    Moreover, it suffices to prove it under the additional assumption that $\psi$ is continuous. With this in hand, we can approximate an arbitrary monotone function $\psi\in L^2(a,b)$ by a sequence $\psi_n$ of monotone continuous functions in $L^2(a,b)$; then $F_\psi^{a,b}$ is convex as the pointwise limit of a sequence of convex functions $F_{\psi_n}^{a,b}$.
    
    For a continuous $\psi$ we can differentiate $F$ directly:
    $$
    F'(t) = \int_a^{a+t}\psi^2(\tau)\, d\tau + t\psi^2(a+t) - 2 \psi(a+t)\int_a^{a+t} \psi(\tau)\, d\tau =  \int_a^{a+t}(\psi(\tau) - \psi(a+t))^2\,d\tau.
    $$
    Monotonicity of $\psi$ implies that $F'$ is increasing. Therefore, $F$ is convex. 
\end{proof}

\begin{proof}[Proof of Proposition~\ref{pr00}]
Without loss of generality, we assume $I=[0,1]$.
Fix any $t \in (0,1)$. We want to prove that the function $A\colon\tau \mapsto \tau^2\xi_\psi^2(\tau)$ has a local supporting line from below at the point $\tau = t$. The supremum in~\eqref{eq3801} is attained on some interval $J_t\subset I$. If this interval is not unique, we choose $J_t$ to be the interval with the maximal length. 

If $|J_t|<t$, then 
$\sup\{\av{\psi^2}J\!-\av{\psi}J^2\colon\! |J|=t\}\!<\!\xi^2_\psi(t).$ 
Hence, $\sup\{\av{\psi^2}J\!-\av{\psi}J^2\colon\! |J|=\tau\}\!<\!\xi^2_\psi(t)$ for all $\tau$ in some neighborhood of $t.$ Therefore, $\xi_\psi$ is constant and $A$ is convex in the same neighborhood of $t$.

If $|J_t|=t$, write $J_t = [c,d]$, $d-c=t$. Since $t<1$, we have either $c>0$ or $d<1$. Without loss of generality assume $d<1$. Apply Lemma~\ref{lem6} to the function $\psi$ with $[a,b] = [c,1]$. Then for $s$ sufficiently close to $t$ we have 
$$
A(s) = s^2 \xi_\psi^2(s) \geq F_\psi^{a,b}(s), 
$$
and for $s=t$ the equality holds. The function $F_\psi^{a,b}$ is convex; therefore, its supporting line at the point $t$ will also be a supporting line for the function $A$ at the same point. 
\end{proof}

Finally, let us prove Theorem~\ref{ThCor}.
\begin{proof}[Proof of Theorem~\ref{ThCor}]
   Theorem~\ref{th1} implies $\xi_{\phi^*}\leq \xi_\phi$, therefore $t^2\xi_{\phi^*}^2(t)\leq t^2\xi_\phi^2(t)$ for all $t \in [0,|I|]$. From Proposition~\ref{pr00} we know that $t^2\xi_{\phi^*}^2$ is convex, therefore $\phi^* \leq \conv{\phi}{|I|}.$ 
\end{proof}

\section*{Acknowledgments}
The authors are grateful to Dmitriy Stolyarov for several helpful comments, particularly those concerning Theorem~\ref{th0}. We are indebted to Fedor Nazarov for the majorization argument realized in Section~\ref{SecNew}.


\begin{thebibliography}{99}
\bibitem{Ryan1} A. Burchard, G. Dafni, R. Gibara.
Mean oscillation bounds on rearrangements. {\it Trans. Amer. Math. Soc.}, 375 (2022), no. 6, 4429--4444

\bibitem{Ryan2} A. Burchard, G. Dafni, R. Gibara.
Vanishing mean oscillation and continuity of rearrangements.
{\it Advances in Mathematics}, Volume 437, February 2024, 109379



\bibitem{Camp} S. Campanato. Propriet\`a di h\"olderianit\`a di alcune classi di funzioni. {\it Ann. Scuola Norm. Sup. Pisa Cl. Sci.~3}, 17 (1963), pp.~175–-188

\bibitem{Gar} J. Garnett. Bounded analytic functions. Revised first edition. Graduate Texts in Mathematics, 236. Springer, New York, 2007. xiv+459 pp. ISBN: 978-0-387-33621-3

\bibitem{Kisl} S.~Kislyakov, N.~Kruglyak. Extremal problems in interpolation theory, Whitney--Besicovitch coverings, and singular integrals. Mathematical Monographs (New Series), vol. 74. Birkh\"auser/Springer Basel AG, Basel (2013). ISBN: 978-3-0348-0468-4

\bibitem{Klemes} I.~Klemes. 
A mean oscillation inequality.
{\it Proc. Amer. Math. Soc. }, 93 (1985), no. 3, 497–-500

\bibitem{Kor1} A. A. Korenovskii. On the connection between mean oscillation and exact integrability classes of functions, {\it Mat. Sb.}~{\bf 181}:12 (1990), 1721--1727 (in Russian); translated in {\it Math. of the USSR-Sbornik}~{\bf 71}:2 (1992), 561--567

\bibitem{Kor2} A. Korenovskii. Mean oscillations and equimeasurable rearrangements of functions. Lecture Notes of the Unione Matematica Italiana, 4. Springer, Berlin; UMI, Bologna, 2007, 188 pp. ISBN: 978-3-540-74708-6.

\bibitem{Mey} N.G.~Meyers. Mean oscillation over cubes and H\"older continuity. {\it Proc. Amer. Math. Soc.}, Vol.~15, No.~5 (1964), pp.~717--721

\bibitem{Os}  A.~Os\c{e}kowski. Sharp estimates for Lipschitz class. {\it The Journal of Geometric Analysis} 26 (2016), pp. 1346--1369

\bibitem{Sar} D.~Sarason. Functions of vanishing mean oscillation. {\it Trans. Amer. Math. Soc.}, 207 (1975), 391--405

\bibitem{exp-vmo} L.~Slavin, P.~Zatitskii. Exponential estimates for VMO with Campanato norm. Preprint.

\bibitem{Spa} S.~Spanne. Some function spaces defined using the mean oscillation over cubes. {\it
Ann. Scuola Norm. Sup. Pisa Cl. Sci.} (3) 19 (1965), 593--608

\bibitem{studia} D. M. Stolyarov, V. I. Vasyunin, P. B. Zatitskiy.
Monotonic rearrangements of functions with small mean oscillation. {\it
Studia Math}, 231 (2015), no. 3, 257--267

\bibitem{sz}
D. M. Stolyarov, P. B. Zatitskiy. Theory of locally concave functions and its applications to sharp estimates of integral functionals. {\it Advances in Mathematics}, Vol.~291 (2016), pp.~228--273



\end{thebibliography}
\end{document}